\documentclass[11pt]{amsart}
\usepackage{hyperref}
\usepackage{amssymb}
\usepackage{epsfig}
\usepackage[all]{xy}

\usepackage[margin=1.1in]{geometry}

\newtheorem{thm}{Theorem}
\newtheorem{prop}{Proposition}
\newtheorem{cor}{Corollary}

\theoremstyle{definition}
\newtheorem{remark}{Remark}

\def\Gr{{\rm Gr}}
\def\af{{\rm af}}
\def\Z{{\mathbb Z}}
\def\C{{\mathbb C}}
\def\Q{{\mathbb Q}}
\def\Sym{{\rm Sym}}
\def\CC{\mathcal{C}}
\def\TT{\mathcal{T}}

\begin{document}
\title[Affine Schubert classes and Schur positivity]{Affine Schubert classes, Schur positivity, and combinatorial
Hopf algebras}
\author{Thomas Lam}
 \thanks{T.L. was partially supported by NSF grants DMS-0600677 and DMS-0652641.}
\address{Harvard University}%
\email{tfylam@math.harvard.edu}%
\maketitle

\begin{abstract}
We suggest the point of view that the Schubert classes of the affine Grassmannian of a simple algebraic group $G$ should be considered as Schur-positive symmetric functions.  In particular, we give a geometric explanation of the Schur positivity of $k$-Schur functions (at $t = 1$).  We also put this in the context of the theory of combinatorial Hopf algebras.
\end{abstract}

\section{Affine Schubert classes ``are'' Schur-positive symmetric functions}
Let $G$ be a simple and simply-connected complex algebraic group,
and let $\Gr_G$ denote the (ind-scheme) affine Grassmannian of $G$.
Let $W$ denote the Weyl group of $G$ and $W_\af$ denote the affine
Weyl group of $G$.  The homology $H_*(\Gr_G) = H_*(\Gr_G,\Z)$ has a
basis given by the Schubert classes $\{\xi_w \mid w \in W_\af/W\}$
\cite{Kum}.  All (co)homologies have $\Z$-coefficients.

If $\iota: G \to G'$ is an inclusion of algebraic groups then there
is a closed embedding $\iota_{\Gr}: \Gr_G \to \Gr_{G'}$, see for
example \cite[A.5]{Gai}. The results of Kumar and Nori
\cite[Proposition (5)]{KN} imply that the homology class $[X] \in
H_*(\Gr_{G'})$ of any finite-dimensional subvariety $X \subset
\Gr_{G'}$ is a (finite) nonnegative sum $[X] = \sum_w a_w \xi_w$ of
the Schubert classes $\xi_w \in H_*(\Gr_{G'})$.  Applying this to
the image $\iota_{\Gr}(X_v)\subset \Gr_{G'}$ of a Schubert variety
$X_v \subset \Gr_G$, we obtain

\begin{thm}\label{thm:pos}
For any $v \in W_\af/W$, the pushforward $(\iota_{\Gr})_*(\xi_v) \in
H_*(\Gr_{G'})$ of a Schubert class of $\Gr_G$ is a nonnegative
linear combination of Schubert classes $\{\xi_w \mid w \in
W'_\af/W'\}$ of $\Gr_{G'}$.
\end{thm}

This simple observation was obtained as a consequence of discussions
with Mark Shimozono, and is the basis for the current article.  In
the ``limit'' $G = SL(\infty, \C)$ the homology
$H_*(\Gr_{SL(\infty,\C)})$ can be identified with the ring $\Sym$ of
symmetric functions, and the Schubert basis is given by the Schur
functions.  Indeed, $\Gr_{SL(\infty,\C)}$ is homotopy equivalent to
$\Omega SU(\infty)$ (\cite{PS}) and by Bott-periodicity
(\cite{Bott}) also to the classifying space $BU(\infty)$. It is well
known that $H_*(BU(\infty)) \simeq \Sym$.  Picking an embedding $G
\hookrightarrow SL(m,\C) \hookrightarrow SL(\infty,\C)$ one obtains
a map $\eta: H_*(\Gr_G) \to \Sym$, and from Theorem \ref{thm:pos} we
see that ``every affine Schubert class is a Schur-positive symmetric
function''.  This statement is slightly misleading (hence the
quotation marks) because for groups $G$ with torsion, one cannot
always arrange for $\eta$ to be an inclusion (see Remark
\ref{rem:SO}).  Nevertheless, this setup is a potent explanation for
many kinds of Schur-positivity.

In the cases $G = SL(n,\C)$ the homology ring $H_*(\Gr_G)$ is
isomorphic to a subalgebra $\Lambda_{(n)}$ of symmetric functions,
and it was shown in \cite{Lam} that under this isomorphism the
Schubert basis is identified with the $k$-Schur functions
$s^{(k)}_\lambda(X)$ \cite{LLM,LM} (with $k = n-1$ and $t = 1$) of
Lapointe, Lascoux, and Morse. It was conjectured in \cite{LLM} that
the $k$-Schur functions expand positively in terms of $(k+1)$-Schur
functions (called {\it $k$-branching positivity}), and in particular
that they expand positively in terms of Schur functions
(corresponding to $k =\infty$).  These conjectures were motivated by
the Macdonald positivity conjecture.  In Section \ref{sec:based}, we
check that the map on homology induced by the natural inclusions
$\iota: SL(n,\C) \to SL(n+1,\C)$ is the natural inclusion in
symmetric functions, obtaining these conjectures as consequences of
Theorem \ref{thm:pos}. Recently, Assaf and Billey \cite{AB} have
given a combinatorial proof of the Schur positivity conjecture, in
the more general case when the parameter $t$ (not discussed here) is
present. Another combinatorial approach to the $k$-branching
positivity will be given in joint work \cite{LLMS} with Lapointe,
Morse, and Shimozono.

In \cite{LSS}, the homology $H_*(\Gr_{Sp(2n,\C)})$ was also
identified with a ring of symmetric functions, and the Schubert
basis constructed.  The inclusions $Sp(2n,\C) \hookrightarrow
Sp(2n+2,\C)$ and $Sp(2n,\C) \hookrightarrow SL(2n,\C)$ give rise to
other branching positivity and Schur-positivity statements.

Our point of view so far puts symmetric functions as the ``universal
target'' for the homologies of affine Grassmannians.  It is natural
to ask how much flexibility there is with the maps $\eta: H_*(\Gr_G)
\to \Sym$ from our homology rings to symmetric functions. In Section
3, we connect this question with the category of combinatorial Hopf
algebras studied by Aguiar, Bergeron, and Sottile \cite{ABS}.  It is
shown in \cite{ABS} that the Hopf algebra of symmetric functions is
the terminal object in the category of cocommutative graded Hopf
algebras equipped with a character.  All the homologies
$H_*(\Gr_{G})$ are cocommutative graded Hopf algebras, so our aim to
write affine Schubert classes as symmetric functions is very natural
from this perspective.  In Section \ref{sec:top}, we give a
topological interpretation of the theorem of \cite{ABS}, replacing
Hopf algebras with $H$-spaces, and the character with a
$U(\infty)$-bundle.

\medskip

{\bf Acknowledgements.} This work arose from my collaboration with
Luc Lapointe, Jennifer Morse, and especially with Mark Shimozono.  I
also thank Nolan Wallach for an interesting discussion.

\section{Based loop spaces and symmetric functions}\label{sec:based}
Let $K \subset G$ denote the maximal compact subgroup.  Let $LG=
G(\C[z,z^{-1}])$ denote the space of algebraic maps $\C^* \to G$,
and let $L^+G = G(\C[z])$ denote the space of algebraic maps $\C \to
G$.  Let $\Omega K \subset LG$ be the space of polynomial based
loops into $K$: these are maps $f: \C^* \to G$ whose restriction to
$S^1 \subset \C^*$ lie in $K$, and such that $f(1) = 1$.  The affine
Grassmannian $\Gr_G$ can be presented as $LG/L^+G$, and it is shown
in \cite[Section 8]{PS} that $\Omega K$ can be identified with
$\Gr_G$ as topological spaces.  In particular, we have $H_*(\Omega
K) \simeq H_*(\Gr_G)$.  Note that our affine Grassmannian, which is
an ind-scheme, is denoted $\Gr_0^{\mathfrak g}$ in \cite{PS}.

Now let $\iota: G \to G'$ be an inclusion.  We will assume that the
maximal compact subgroups are chosen so that $\iota(K) \subset K'$.
To calculate the map $(\iota_{\Gr})_*: H_*(\Gr_G) \to
H_*(\Gr_{G'})$, we will instead consider the map $H_*(\Omega K) \to
H_*(\Omega K')$.  We may do so because the identification $\Omega K
\simeq LG/L^+G = \Gr_G$ is induced by the inclusion $\Omega K
\subset LG$, and we have the commutative diagram
\begin{align*}
\xymatrix{%
\Omega K \ar[d] \ar[r] & \Omega K' \ar[d] \\
LG \ar[r] & LG' }
\end{align*}
Since $\Omega K \to \Omega K'$ is a map of groups, the map
$(\iota_{\Gr})_*: H_*(\Gr_G) \to H_*(\Gr_{G'})$ is a Hopf-morphism.

The maximal compact subgroup of $SL(n,\C)$ is the special unitary
group $SU(n)$.  We shall always consider $Sp(2n,\C) \subset
SL(2n,\C)$ in the natural way, and the maximal compact subgroup is
denoted $Sp(n) = SU(2n) \cap Sp(2n,\C)$.  We shall consider three
inclusions: (a) $SL(n,\C) \hookrightarrow SL(n+1,\C)$ giving $SU(n)
\hookrightarrow SU(n+1)$, (b) $Sp(2n,\C) \hookrightarrow
Sp(2n+2,\C)$ giving $Sp(n) \hookrightarrow Sp(n+1)$, and (c)
$Sp(2n,\C) \hookrightarrow SL(2n,\C)$ giving $Sp(n) \hookrightarrow
SU(2n)$.  The inclusions (a) and (b) are the natural ones, induced
by inclusions of coordinate subspaces. We shall assume that these
embeddings are compatible, that is, the two compositions $Sp(2n,\C)
\hookrightarrow Sp(2n+2,\C) \hookrightarrow SL(2n+2,\C)$ and
$Sp(2n,\C) \hookrightarrow SL(2n,\C) \hookrightarrow SL(2n+2,\C)$
are identical.  This is easy to achieve (see \cite[p.183]{MT}).

We now switch to the notation $H_*(\Omega K)$, instead of
$H_*(\Gr_G)$.  The homology ring $H_*(\Omega SU(n))$ (resp.
$H_*(\Omega Sp(n))$ is a polynomial algebra with generators in
dimensions $2,4,\ldots,2n-2$ (resp. $2,6,\ldots,4n-2$).

\begin{prop}\label{prop:inc}
The induced maps on homology
\begin{align*}
H_*(\Omega SU(n)) & \to H_*(\Omega SU(n+1)) \\
H_*(\Omega Sp(n)) & \to H_*(\Omega Sp(n+1)) \\
H_*(\Omega Sp(n)) &\to H_*(\Omega SU(2n))
\end{align*}
 are Hopf-inclusions. The
first two maps are $\Z$-module isomorphisms below degrees $2n-1$ and
$4n-1$ respectively.
\end{prop}
\begin{proof}
That $H_*(\Omega SU(n)) \to H_*(\Omega SU(n+1))$ is injective is
shown in \cite[Proposition 8.4]{Bott}, by noting that the fiber and base of the sequence
$\Omega SU(n) \to \Omega SU(n+1) \to \Omega S^{2n-1}$ has homology concentrated in even dimensions.  The same argument, with
$\Omega S^{4n-1}$ replacing $\Omega S^{2n-1}$ shows that $H_*(\Omega
Sp(n)) \to H_*(\Omega Sp(n+1))$ is injective.  The claims concerning
isomorphisms in low dimensions also follow from this calculation.

Since all the inclusions are compatible, to show that $H_*(\Omega
Sp(n)) \to H_*(\Omega SU(2n))$ is injective it suffices to show that
$H_*(\Omega Sp(\infty)) \to H_*(\Omega SU(\infty))$ is injective. The
corresponding cohomology map $H^*(\Omega SU(\infty)) \simeq
H^*(BU(\infty)) \to H^*(Sp(\infty)/U(\infty)) \simeq H^*(\Omega
Sp(\infty))$ is calculated in \cite[3.18]{MT}, and is manifestly
surjective (see also \cite[Lemma 5.2 and Theorem 5.14]{MT}).  Thus
the dual map in homology is injective.
\end{proof}

\begin{remark}\label{rem:SO}
For groups $G$ with torsion, it may not be possible to find a
Hopf-inclusion $H_*(\Omega K) \to H_*(\Omega SU(m))$.  Let $K =
SO(n)$ be the special orthogonal group.  $SO(n)$ is not
simply-connected, but we have $\Omega_0 SO(n) \simeq \Omega Spin(n)$
where $\Omega_0 SO(n)$ denotes the connected component of the
identity.  According to \cite[Proposition 10.1]{Bott}, the Hopf
algebra $H^*(\Omega_0 SO(4n))$ has a primitive subspace of dimension
2 in degree $4n-2$, so admits no Hopf-inclusion into $H_*(\Omega
SU(m))$, which has primitive subspaces of at most dimension 1 in
each degree.
\end{remark}

We now calculate the inclusions of Proposition \ref{prop:inc} in
terms of symmetric functions.  Presumably the following calculations
are well-known to topologists; we present them in a form emphasizing
the connections with symmetric functions.

For symmetric function definitions we use here, we refer the reader
to \cite{Lam,LSS}.  Let $\Sym$ denote the ring of symmetric
functions in infinitely many variables $x_1,x_2,\ldots$, over $\Z$.
We let $h_i$ denote the homogeneous symmetric functions, $e_i$
denote the elementary symmetric functions, and $p_i$ denote the
power sum symmetric functions.  We let $s_\lambda$ denote a Schur
function. We let $\omega: \Sym \to \Sym$ denote the conjugation
involution of $\Sym$, sending $h_i$ to $e_i$.  The comultiplication
of $\Sym$ is given by $\Delta(h_i) = \sum_{j=0}^i h_j \otimes
h_{i-j}$, where $h_0 := 1$, or by $\Delta(p_i) = 1 \otimes p_i + p_i
\otimes 1$.

\subsection{$k$-branching in $SL(n,\C)$}\label{sec:SLk}
The Hopf-subalgebra $\Z[h_1,h_2,\ldots,h_{n-1}]$ is Hopf-isomorphic
to $H_*(\Omega SU(n))$ and under this isomorphism we showed in
\cite{Lam} that the Schubert basis $\{\xi_w \in H_*(\Omega SU(n))
\mid w \in W_\af/W\}$ is identified with the $k$-Schur functions
$s_\lambda^{(k)} \in \Sym$ of \cite{LM}, with $k = n-1$. For $k$
larger than the degree, a $k$-Schur function is simply a Schur
function.  There is some flexibility in this isomorphism: one may
compose with the involution $\omega$, which sends $k$-Schur
functions to $k$-Schur functions. At the level of the Weyl group,
this corresponds to the non-trivial Dynkin diagram automorphism of
$\tilde{A}_{n-1}$ which fixes the affine node $0$. Note that the
degree of a homogeneous symmetric function is half the topological
degree of the corresponding homology class.

Let $$\phi: \Z[h_1,h_2,\ldots,h_{n-1}] \simeq H_*(\Omega SU(n))
\hookrightarrow H_*(\Omega SU(n+1)) \simeq \Z[h_1,h_2,\ldots,h_n]$$
be the Hopf-inclusion induced by Proposition \ref{prop:inc}.  Since
$H_2(\Omega SU(n+1))$ has rank 1, we must have by Theorem
\ref{thm:pos} and Proposition \ref{prop:inc} $\phi(h_1) = h_1$.
 Since $\phi$ is a Hopf-morphism it must send primitive elements to
primitive elements.  The primitive elements are exactly the power
sum symmetric functions $p_1,p_2,\ldots$.  Since $\phi$ is an
isomorphism in low dimensions, it must send each power sum symmetric
function $p_i$ ($1 \leq i \leq n-1$) to $\pm p_i$.  We have two
choices $\phi(p_2) = \pm p_2$.  We may assume, by possibly composing
the identification $H_*(\Omega SU(n+1)) \simeq
\Z[h_1,h_2,\ldots,h_n]$ with $\omega$, that $\phi(p_2) = p_2$.  Now
suppose that we have established $\phi(p_i) = p_i$ for all $i < j$,
for some $j
> 2$. Expressing $e_j$ as a polynomial in power sum symmetric
functions, $p_j$ occurs with coefficient $(-1)^j p_j/j$. We see that
$e_j \pm 2p_j/j$ has monomials with fractional coefficients, so does
not lie in $\Z[h_1,h_2,\ldots,h_n]$, and conclude that $\phi(p_j) =
p_j$. It follows by induction that $\phi$ is the obvious inclusion.

From Theorem \ref{thm:pos}, we obtain
\begin{cor}
Every $k$-Schur function is $(k+1)$-Schur positive.  In particular,
$k$-Schur functions are Schur positive.
\end{cor}

\subsection{$k$-branching in $Sp(2n,\C)$}
Let $P_i\in \Sym$ denote the Schur $P$-functions labeled by a single
row. In \cite{LSS}, we showed that one has a Hopf-isomorphism
$H_*(\Omega Sp(n)) \simeq \Z[P_1,P_3,\ldots,P_{2n-1}]$, identifying
the homology Schubert basis $\{\xi_w \in H_*(\Omega Sp(n)) \mid w
\in W_\af/W\}$ with symmetric functions denoted $P^{(n)}_w$ (type
$C$ $k$-Schur functions).  The following result is implicit, but not
completely spelt out in \cite{LSS}, so we do so here:

\begin{prop}
Let $w \in W_\af$ be a minimal coset representative in $W_\af/W$.
 For $n > \ell(w)$, the symmetric function $P^{(n)}_w$ is a Schur
$P$-function.
\end{prop}
\begin{proof}
We preserve all the notation of \cite{LSS}.  Using duality, it
suffices to show that the symmetric functions $Q^{(n)}_w$ of
\cite{LSS} coincide with the Schur $Q$-functions when $n$ is large.
In this case, a reduced expression of $w$ does not involve the
simple generator $s_n$.  (The simple generators of $W_\af$ are
$s_0,s_1,\ldots,s_n$, where $s_0$ is the affine node.)  Thus in the
formula \cite[(1.2)]{LSS} for $Q^{(n)}_w$, all $v \in W_\af$
involving $s_n$ may be ignored.  Comparing the definition of $Z$-s
in \cite{LSS} with \cite[(4.1)]{FK}, we see that our $Q^{(n)}_w$ are
a power of 2 times the $B_n$-Stanley symmetric functions of
\cite{FK}.  The latter are known to be Schur $P$-functions
\cite[Theorem 8.2]{FK} in the special case of a Grassmannian
$B_n$-element.  It follows that the $Q^{(n)}_w$ of \cite{LSS} are
Schur $Q$-functions.
\end{proof}

Note that one has $\Q[P_1,P_3,\ldots,P_{2n-1}] \simeq
\Q[p_1,p_3,\ldots,p_{2n-1}]$. We have the formula
\begin{equation}
\label{E:P} P_i = \frac{1}{2}\sum_{j=0}^i e_jh_{i-j}
\end{equation}
which gives a symmetric function with integral coefficients, despite
the half. When written as a polynomial in power sum symmetric
functions, only the terms $e_i$ and $h_i$ in \eqref{E:P} involve
$p_i$.  So it is not difficult to deduce that for an odd integer
$i$, the coefficient of $p_i$ in the expansion of $P_i$, when
expressed as a polynomial in (odd) power sum symmetric functions, is
equal to $1/i$.

The primitive subspace of $\Z[P_1,P_3,\ldots,P_{2n-1}]$ is spanned
by $p_1,p_3,\ldots,p_{2n-1}$.  It follows from Proposition
\ref{prop:inc} and the same argument as in Section \ref{sec:SLk}
(but without the complication of the conjugation) that the map
$$\psi: \Z[P_1,P_3,\ldots,P_{2n-1}] \simeq H_*(\Omega Sp(n))
\hookrightarrow H_*(\Omega Sp(n+1)) \simeq
\Z[P_1,P_3,\ldots,P_{2n+1}]$$ is the natural inclusion.  From
Theorem \ref{thm:pos}, we obtain
\begin{cor}
The symmetric function $P^{(n)}_w$ expands positively in terms of
$\{P^{(n+1)}_v\}$. In particular, $P^{(n)}_w$ is positive in terms
of Schur $P$-functions.
\end{cor}

\subsection{$Sp(2n,\C)$ to $SL(2n,\C)$ branching}
Let $$\kappa: \Z[P_1,P_3,\ldots,P_{2n-1}]\simeq H_*(\Omega Sp(n))
\hookrightarrow H_*(\Omega SU(2n)) \simeq
\Z[h_1,h_2,\ldots,h_{2n-1}]$$ be the Hopf-inclusion induced by
Proposition \ref{prop:inc}.  Since all our inclusions commute, to
show that $\kappa$ is the natural inclusion of symmetric functions,
it suffices to show that $\kappa_\infty:\Z[P_1,P_3,\ldots]\simeq
H_*(\Omega Sp(\infty)) \hookrightarrow H_*(\Omega SU(\infty)) \simeq
\Sym$ is the natural inclusion. Let $\Gamma^* \subset \Sym$ denote
the ring of symmetric functions dual to $\Z[P_1,P_3,\ldots]$,
considered in \cite{LSS}.  We have $\Gamma^* = \Z[Q_1,Q_3,\ldots]$,
where $Q_i = 2P_i$.  The relations satisfied by $Q_i$ can be deduced
from \cite[(2.8)]{LSS}.

The dual map $\theta: H^*(\Omega SU(\infty)) \twoheadrightarrow
H^*(\Omega Sp(\infty))$ is given explicitly in \cite[3.18]{MT},
where $H^*(\Omega SU(\infty))$ is presented as $\Z[h_1,h_2,\ldots]$
and $H^*(\Omega Sp(\infty))$ is presented as $\Z[Q_1,Q_3,\ldots]$,
and the map is the surjection given by $\theta(h_i) = Q_i$ (the
Schur $Q$-function $Q_i$ is defined for even $i$ as well).  In terms
of power sum symmetric functions, this map is given by
$\theta(p_{2i}) = 0$ and $\theta(p_{2i+1}) = 2p_{2i+1}$.  Thus the
dual $\kappa^*_\infty$ of our desired map differs from the map
$\theta$ by Hopf-automorphisms: so we must have
$\kappa^*_\infty(p_j) = \pm\theta(p_j)$.  Taking duals and using
\cite[Lemma 2.1]{LSS}, which roughly says that $\theta: \Sym \to
\Gamma^*$ and the inclusion $\Z[P_1,P_3,\ldots,P_{2n-1}] \subset
\Sym$ are adjoint, we see that we have $\kappa_\infty(p_j) = \pm
p_j$. To see that the sign is positive, we may argue as in Section
\ref{sec:SLk}, and deduce that $\kappa_\infty$ itself is the natural
inclusion of rings of symmetric functions. From Theorem
\ref{thm:pos}, we obtain

\begin{cor}
The symmetric function $P^{(n)}_w$ expands positively in terms of
$(2n-1)$-Schur functions.
\end{cor}

With $n \to \infty$, we obtain as a special case the well-known fact
that Schur $P$-functions are Schur positive.

\section{Combinatorial versus topological Hopf
algebras}\label{sec:top} Let $H$ be a graded, connected, Hopf
algebra defined over $\Z$.  A character $\chi: H \to \Z$ is a
morphism of $\Z$-algebras.  Aguiar, Bergeron, and Sottile \cite{ABS}
define a {\it combinatorial Hopf algebra} to be a pair $(H,\chi)$.
(The Hopf algebras of \cite{ABS} are in fact over a field and we
have changed to the integers.)  The category $\CC$ of combinatorial
Hopf algebras has arrows $g: (H, \chi) \to (H',\chi')$ given by
Hopf-morphisms $g: H \to H'$ such that $\chi = \chi' \circ g$.
Symmetric functions $\Sym$ have a canonical character, given by
$\chi_\Sym(h_i) = 1$, or $f(x_1,x_2,x_3,\ldots) \mapsto
f(1,0,0,\ldots)$.

Aguiar, Bergeron, and Sottile show
\begin{thm}\label{thm:ABS}
The terminal object of the category of cocommutative Hopf algebras
is $(\Sym,\chi_\Sym)$.
\end{thm}
From this point of view, the central thesis of this article, which
is to express affine Schubert classes as symmetric functions, is
very natural.  The Hopf algebras considered in \cite{ABS} have a
combinatorial origin, while the Hopf algebras considered in the
present article have a topological origin.  It is thus natural to
find the topological version of Theorem \ref{thm:ABS}.

By a $H$-space we will mean a connected topological space $X$, which
is equipped with a homotopy-associative multiplication $m: X \times
X \to X$. A map of $H$-spaces is one that commutes with
multiplication up to homotopy. Let us consider the category $\TT$
whose objects are pairs $(X, E)$ where $X$ is a $H$-space and $E$ is
a $U(\infty)$-bundle over $X$ such that the following diagram is
Cartesian:
\begin{equation}\label{eq:pullback}
\xymatrix{
E \times E \ar@{->}[r] \ar@{->}[d] & E \ar@{->}[d]  \\
X \times X \ar@{->}[r] & X  }.
\end{equation}
The maps in $\TT$ are given by homotopy classes of $H$-space maps
which induce Cartesian diagrams.  Let us denote the total Chern
class of $U(\infty)$-bundle by $c(E)$.  The element $c(E)$ lies in
the completion of the cohomology $H^*(X)$, and gives rise to a
linear map $\chi_E: H_*(X) \to \Z$.  The condition that $\chi_E$ is
a character is that $c(E)$ is grouplike: $m^*(c(E)) = c(E) \otimes
c(E)$.  This condition follows from the diagram \eqref{eq:pullback}.

The classifying space $BU(\infty)$ has a natural $U(\infty)$-bundle
$EU(\infty)\to BU(\infty)$.  The $H$-space structure $m:BU(\infty)
\times BU(\infty) \to BU(\infty)$ classifies the Whitney sum of
vector bundles. Thus if $E \to X$ and $E' \to X$ are classified by
$f:X \to BU(\infty)$ and $f':X \to BU(\infty)$ then $E \oplus E'$ is
classified by $m \circ (f,f'): X \to BU(\infty) \times BU(\infty)
\to BU(\infty)$. It follows that $(BU(\infty),EU(\infty))$ satisfies
\eqref{eq:pullback}.  We shall pick an isomorphism $H_*(BU(\infty))
\simeq \Sym$ so that $c(EU(\infty)) = 1 + h_1 + h_2 + \cdots$.  (The
usual isomorphism would give the elementary symmetric functions, so
we compose with $\omega$.)  The following result follows nearly
immediately from the definitions.

\begin{thm}\label{thm:newABS}
The pair $(BU(\infty),EU(\infty))$ is the terminal object in the
category $\TT$.  There is a functor $\TT \to \CC$ given by $(X,E)
\mapsto (H_*(X),\chi_E)$, sending $(BU(\infty),EU(\infty))$ to
$(\Sym,\chi_\Sym)$.
\end{thm}
\begin{proof}
Let $(X,E) \in \TT$.  Let $f: X \to BU(\infty)$ be the (unique up to
homotopy) map classifying the bundle $E \to X$.  Then $f \times f :
X\times X \to BU(\infty) \times BU(\infty)$ classifies $E \times E$
and the diagram
\begin{equation}
\xymatrix{
X \times X \ar@{->}[r] \ar@{->}[d] & BU(\infty) \times BU(\infty) \ar@{->}[d]  \\
X \ar@{->}[r] & BU(\infty)  }
\end{equation}
is homotopy commutative by \eqref{eq:pullback} for $(X,E)$ and for
$(BU(\infty),EU(\infty))$. Thus the map $X \to BU(\infty)$ is a
$H$-space map and $f:(X,E) \to (BU(\infty),EU(\infty))$ a morphism
in $\TT$. The last statement has already been established.
\end{proof}

\begin{remark}\def\QSym{{\rm QSym}}
The terminal object in the category $\CC$ of all combinatorial Hopf
algebras is the Hopf algebra of quasisymmetric functions $\QSym$. It
is shown in \cite{BR} that $\QSym \simeq H^*(\Omega \Sigma {\mathbb
CP}^\infty)$, where $\Sigma$ denotes suspension.  Thus there should
be a different version of Theorem \ref{thm:newABS} with $\Omega
\Sigma {\mathbb CP}^\infty$ as the terminal object.
\end{remark}


\begin{thebibliography}{AAAA}

\bibitem[ABS]{ABS} {\sc M.~Aguiar, N.~Bergeron, and F.~Sottile:}
Combinatorial Hopf algebras and generalized Dehn-Sommerville
relations, {\sl Compos. Math.} {\bf 142} (2006),  no. 1, 1--30.

\bibitem[AB]{AB} {\sc S.~Assaf and S.~Billey:} in preparation.

\bibitem[BR]{BR}
{\sc A. Baker and B.~Richter:} Quasisymmetric functions from a
topological point of view, {\sl Math. Scand.} \textbf{103} (2008),
208–-242.

\bibitem[Bott]{Bott} {\sc R.~Bott:}
The space of loops on a Lie group, {\sl Michigan Math. J.}
\textbf{5} (1958), 35--61.

\bibitem[FK]{FK} {\sc S.~Fomin and A.N.~Kirillov:} Combinatorial $B_n$-analogues of Schubert
polynomials, {\sl Trans. Amer. Math. Soc.} \textbf{348} (1996),
3591--3620.

\bibitem[Gai]{Gai} {\sc D.~Gaitsgory:} Construction of central elements in the
affine Hecke algebra via nearby cycles, {\sl Invent. Math.}
\textbf{144} (2001), 253--280.

\bibitem[Kum]{Kum} {\sc S.~Kumar:}  Kac-Moody groups, their flag
varieties and representation theory, {\sl Progress in Mathematics}
\textbf{204} Birkh\"{a}user Boston, Inc., Boston, MA, 2002.

\bibitem[KN]{KN} {\sc S.~Kumar and M.V.~Nori:}
Positivity of the cup product in cohomology of flag varieties
associated to Kac-Moody groups, {\sl Internat. Math. Res. Notices}
1998, no. 14, 757--763.

\bibitem[Lam]{Lam} {\sc T.~Lam:} Schubert polynomials for the affine
Grassmannian, {\sl J. Amer. Math. Soc.} \textbf{21} (2008),
259--281.

\bibitem[LLMS]{LLMS} {\sc T.~Lam, L.~Lapointe, J.~Morse,
and M.~Shimozono:} in preparation.

\bibitem[LSS]{LSS} {\sc T.~Lam, A.~Schilling, and M.~Shimozono:}
Schubert polynomials for the affine Grassmannian of the symplectic
group {\sl Math. Z.}, to appear; {\tt arXiv:0710.2720}.

\bibitem[LLM]{LLM} {\sc L.~Lapointe, A.~Lascoux, and J.~Morse:}
Tableau atoms and a new Macdonald positivity conjecture, Duke Math.
J. \textbf{116} (2003),  no. 1, 103--146.

\bibitem[LM]{LM} {\sc L. Lapointe and J. Morse:}
Tableaux on $k+1$-cores, reduced words for affine permutations, and
$k$-Schur expansions, {\sl J. Combin. Theory Ser. A} \textbf{112}
(2005), no. 1, 44--81.

\bibitem[MT]{MT} {\sc M.~Mimura and H.~Toda:}
Topology of Lie groups. I, II. Translated from the 1978 Japanese
edition by the authors, {\sl Translations of Mathematical
Monographs} 91. American Mathematical Society, Providence, RI, 1991.
iv+451 pp.

\bibitem[PS]{PS} {\sc A.~Pressley and G.~Segal:} Loop groups, Clarendon
Press, Oxford, 1986.
\end{thebibliography}
\end{document}